\documentclass[12pt]{amsart}
\usepackage{amsmath,amssymb,amsbsy,amsfonts,amsthm,latexsym,mathabx,
            amsopn,amstext,amsxtra,euscript,amscd,stmaryrd,mathrsfs,
            cite,array,mathtools,enumerate}
\usepackage{url}
\usepackage[colorlinks,linkcolor=blue,anchorcolor=blue,citecolor=blue,backref=page]{hyperref}
\usepackage{color}
\usepackage{graphics,epsfig}
\usepackage{graphicx}
\usepackage{float}

\usepackage{mathtools}
\usepackage{todonotes}

\hypersetup{breaklinks=true}

\RequirePackage{mathrsfs} \let\mathcal\mathscr

 \usepackage[norefs,nocites]{refcheck}
 
\usepackage[all]{xy}

\newtheorem{theorem}{Theorem}
\newtheorem{thm}[theorem]{Theorem}
\newtheorem{lem}[theorem]{Lemma}

\newtheorem{cor}[theorem]{Corollary}

\numberwithin{equation}{section}
\numberwithin{theorem}{section}
\numberwithin{table}{section}

\newfont{\teneufm}{eufm10}
\newfont{\seveneufm}{eufm7}
\newfont{\fiveeufm}{eufm5}
\newfam\eufmfam
     \textfont\eufmfam=\teneufm
\scriptfont\eufmfam=\seveneufm
     \scriptscriptfont\eufmfam=\fiveeufm
\def\a{\alpha}

\def\eqref#1{(\ref{#1})}

\def\cX{{\mathcal X}}

\def\ep{{\mathbf{\,e}}_p}

\def\({\left(}
\def\){\right)}
\def\l|{\left|}
\def\r|{\right|}

\def\mand{\qquad \mbox{and} \qquad}

\theoremstyle{definition}

 \def \F{{\mathbb F}}

\def \Z{{\mathbb Z}}

	\def \l{\lambda}
	
\begin{document}
\title{Binomial exponential sums}

\author{Igor E. Shparlinski}
\address{School of Mathematics and Statistics, University of New South Wales,
 Sydney NSW 2052, Australia}
\email{igor.shparlinski@unsw.edu.au}

\author{Jos\'e Felipe Voloch}

\address{
School of Mathematics and Statistics,
University of Canterbury,
Private Bag 4800, Christchurch 8140, New Zealand
}
\email{felipe.voloch@canterbury.ac.nz}

\begin{abstract} We obtain new bounds of exponential sums
modulo a prime $p$  with binomials  $ax^k + bx^n$. In particular, for $k=1$,
we improve  the bound of  Karatsuba (1967) 
from  $O(n^{1/4} p^{3/4})$ to $O\(p^{3/4}  +    n^{1/3}p^{2/3}\)$ for any $n$, 
and then use it to improve the bound of Akulinichev (1965)
from $O(p^{5/6})$ to $O(p^{4/5})$ for   $n | (p-1)$. The result is based on a new bound
on the number of solutions and of degrees of irreducible components of certain equations 
over finite fields. 
\end{abstract}


\maketitle

\section{Introduction}

\subsection{Background and motivation}
For a prime $p$ we consider the binomial exponential sums 
$$
S_{k,n}(a,b) =  \sum_{x=0}^{p-1} \ep\(ax^k + bx^n\)
$$ 
(where $\ep\(x\) = e^{2\pi i x/p}$) with positive integers $k$ and $n$ and arbitrary integer coefficients $a$ and $b$.

There are several bounds and applications of such sums which go beyond the classical Weil bound, 
see~\cite{Aku,BCPP1,BCPP2,CP03,CP10,CP11,Kar,Yu} and references therein.  In particular, bounds for such 
binomial sums played a key role in the approach and resolution in~\cite{BCPP1,CoKo,CP11,GKMS}
to   the conjecture of Goresky and Klapper~\cite{GK} and   in the closely related generalised 
 Lehmer conjecture~\cite{BCPP2}; for very recent development and 
 generalisations see~\cite{ACMPPRT, CMPR}. 
 
A standard technique relates bounding these sums to bounding the number of solutions of certain equations 
over finite fields. Previous papers have used the Weil bound, see~\cite{Lor}, when applicable as well as 
elementary bounds coming from Bezout's theorem, in the range where Weil's bound becomes 
trivial, to bound the number of solutions of these equations. Here, we obtain sharper bounds to
the number of solutions of these equations. 

The novelty of our approach consists of a combination of two ideas. First, we use the 
method of~\cite{SV} (and particularly the explicit version for plane curves from~\cite{V}) that give improvements of
the Weil bound for large degrees. Second, and perhaps more importantly, the equations we need
to study are sometimes not irreducible and we need to bound from below the degrees of their irreducible components
and consequently the number of these components.
This is achieved by using ABC-type bounds for solutions of
equations over \emph{function fields} using the methods of~\cite{Vabc}. A connection between irreducible factors and the polynomial 
ABC-results for $f(X)-f(Y)$ where $f(X)$ is a one-variable sparse polynomial in characteristic zero, is due  to 
Zannier~\cite{Z}. 
We have transposed this technique to positive characteristic for the same kind of polynomials in~\cite{ShV}. Here we extend
this to a wider class of polynomials while sharpening the method and give applications to new bounds
of binomial exponential sums $S_{k,n}(a,b)$.  We expect this method to have wider applications.

\subsection{Set-up and some previous  results}

Define 
$$
M_{k,n} = \max_{\substack{a, b \in \Z\\ \gcd(ab,p)=1}} \left|S_{k,n}(a,b)\right|. 
$$

In the special case $k=1$ we set 
$$
M_{n} = M_{1,n}.
$$

We also  recall the bound of 
Karatsuba~\cite[Theorem~1]{Kar}
\begin{equation}
\label{eq:Kar}
M_n \le (n-1)^{1/4} p^{3/4}, 
\end{equation}
which holds for any $n\ge 1$. 
Furthermore, Akulinichev~\cite[Theorem~1]{Aku} have shown that 
\begin{equation}
\label{eq:Aku}
M_n \le p/\sqrt{\gcd(n, p-1)}. 
\end{equation}
In particular combining~\eqref{eq:Kar} and~\eqref{eq:Aku}  we see that if $n \mid p-1$ then 
\begin{equation}
\label{eq:Aku 56}
M_n \le p^{5/6}.
\end{equation}
see~\cite[Corollary]{Aku}.
Here we improve~\eqref{eq:Kar}  for an arbitrary $n$ and in then use it to  improve~\eqref{eq:Aku 56} 
and obtain 
$$
M_n = O( p^{4/5}).
$$
in the case when $n \mid p-1$, see  Corollary~\ref{cor:Mn_p-1} below. 

Most of the above results are based on new upper bounds on the number $T_{k,n}$ of solutions to the system of 
equations 
$$
u^k+v^k = x^k+y^k \mand u^n+v^n = x^n+y^n, \qquad u,v,x,y \in \F_p, 
$$
over the finite field $\F_p$ of $p$ elements.  As before, in the special case $k=1$ we define
$$
T_{n} = T_{1,n}.
$$

  For example, Bourgain,   Cochrane,   Paulhus and  Pinner~\cite[Theorem~3]{BCPP1} have shown
that if
$$
\gcd(n,p-1) = 1 \mand  \gcd(n-1, p-1) \le \frac{9}{50} p^{16/23}
$$
then 
\begin{equation}
\label{eq:BCPP}
T_n \le 13658 p^{66/23}.
 \end{equation}
 Cochrane  and   Pinner~\cite[Theorem~7.1]{CP11} have sharpened the constant 
 in the bound~\eqref{eq:BCPP} and also extended it to $T_{k,n}$.  
 
 Here,  in Section~\ref {sec:UB-zero}, we obtain  new 
 bounds. In particular, for $k=1$ we  improve in a wide range the trivial bound 
$T_n = O(np^2)$ (used in~\cite{Kar}). This bound is based on the investigation of irreducible factors of the polynomial
\begin{equation}
\label{eq:Poly F_n}
F_n(X,Y) = X^n+Y^n-(X+Y-1)^n-1\in \F_p[X,Y],
 \end{equation}
  which could be of independent interest, and also an application of some ideas 
and results from~\cite{SV,Vabc,V}.

\subsection{Notation} 
We recall that the notations $U = O(V)$,  $U \ll V$ and  $V \gg U$,   are all 
equivalent to the statement that $|U| \le cV$ for some constant $c$, which is absolute throughout this work.

The letters $k$ and $n$ always denote integer numbers and the letter $p$ always denotes a prime. 

\section{Factors and zeros of some polynomials}

\subsection{Lower bounds on the degree of irreducible factors}
We use some basic facts about the divisors on curves, which can be found in~\cite{Lor}. 
\begin{lem}
\label{degx}
Let $\cX$ be the smooth projective model of a plane curve $h(x,y)=0$ of degree $d$ such that
the homogeneous term of degree $d$ of $h$ is not divisible by $x$ or $y$. Then 
$x$ has degree $d$ as a function on $\cX$. 
\end{lem}

\begin{proof} 
The poles of $x$ and $y$ are among the branches above the points at infinity of the plane curve $h=0$ and
these points at infinity correspond to factors $x-\a y$ with $\a \ne 0$ of the homogeneous term of degree $d$ of $h$,
by the hypothesis. The function $x-\a y$ vanishes at the corresponding branches so if $x$ has a pole at such a
branch, $y$ also has a pole there of the same order and vice versa. So $x$ and $y$ have the same polar divisor $D$.
The functions $x^iy^j$, $i+j \le m$, belong to the Riemann-Roch space $H^0(mD)$, see~\cite[pg.~306]{Lor} and 
the linear relations among them
come from multiples of $h$, so a standard calculation~\cite[pg.~329]{Lor}
gives $\dim H^0(mD) \ge md + O(1)$. 
On the other hand, the 
Riemann-Roch theorem, see~\cite[Chapter IX]{Lor}, gives  $\ell(mD) =  m\deg D+ O(1)$ and it follows 
that $\deg x = \deg D \ge d$. But it is
clear that $\deg x \le d$ and this completes the proof.
\end{proof}

We now extend  the definition of the polynomial $F_n(X,Y)$ in~\eqref{eq:Poly F_n} to arbitrary ground fields

\begin{lem}
\label{lem:deg-bound n}
Let $K$ be a field of positive characteristic $p$ and let 
$n<p$.
If $h(X,Y)$ is an irreducible polynomial factor of $F_n(X,Y) \in K[X,Y]$ of degree $d$,
other than $X-1$, $Y-1$, $X+Y$, then
$d \gg \min \{p/n, {n}\}$.
\end{lem}

\begin{proof}
Let $\cX$ be a smooth model of the curve $h=0$. The genus of $\cX$ is at most
$(d-1)(d-2)/2$. On $\cX$, the functions $x,y$ and $x+y-1$ have at most $d$ zeros and $d$
poles (the latter on the line at infinity) so they are $S$-units for some set $S$ 
of places of $\cX$ with $\# S \le 4d$.
Consider the functions $u_1=x^{n},u_2=y^{n},u_3=-(x+y-1)^n$, which are also $S$-units
and satisfy the unit equation $u_1+u_2+u_3=1$.

The $u_i$ are functions on $\cX$ so $(u_1{:}u_2{:}u_{3})$  
defines a morphism $\cX \to \mathbb{P}^{2}$ of
degree at most $dn$. If $dn \ge p$, the 
desired result follows immediately.
If $dn < p$,  then~\cite[Theorem~4]{Vabc} holds with the same proof
in characteristic $p>0$ (as the morphism has classical orders by~\cite[Corollary~1.8]{SV}).
Also $\deg u_1 = nd$ by Lemma \ref{degx} since $h$ satisfies the hypothesis being a factor of
$f(x,y)/((x-1)(y-1))$, so we get
$$
nd \le \deg u_1 \le 3(d(d-3)+4d) \ll d^2 
$$ 
giving the result, provided $u_1,u_2,u_3$ are linearly independent over $K$.

If $au_1+bu_2+cu_3=0$ and $abc \ne 0$, then we consider the unit equation $-au_1/bu_2 -cu_3/bu_2=1$.
We claim that the degree of $-au_1/bu_2$ is $nd$. This follows if we show that the degree of $x/y$ is $d$.
Now, $x$ has $d$ zeros counted with multiplicity, so the same will be true for $x/y$ unless $y$ vanishes in 
one of the zeros of $x$. This does not happen because $f(x,y)$ does not vanish at the origin for $n$ even and
$f(x,y)/(x+y)$ does not vanish at the origin for $n$ odd. So the same argument as before gives the inequality
of the theorem.

If $c=0$ then $u_1/u_2$ is constant so $x/y$ is constant, say $y=\a x$. The equation $f(x,\a x)=0$ has to be
satisfied identically, which means by looking at the linear term that $\a=-1$,  that is, $x+y=0$ and the constant term forces $n$
to be odd. If $a=0$, a similar argument gives $x=1$ and if $b=0$ then $y=1$.
\end{proof}

We  now treat the more general polynomials 
\begin{equation}
\label{eq:Poly F_kn}
F_{k,n}(X,Y) = (X^n+Y^n-1)^{k/r} - (X^k+Y^k-1)^{n/r} \in \F_p[X,Y],
 \end{equation} 
 where $k$ and $n$ are distinct integers and $r=\gcd(k,n)$. They
reduce to $F_n$ when $k=1$.

Unfortunately, the result that we obtain below about the  components of the polynomials~\eqref{eq:Poly F_kn} is weaker than
the corresponding statement for $F_n$. One reason is that Lemma~\ref{degx} does not apply for $k>1$.

\begin{lem}
\label{lem:deg-bound kn}
Let $K$ be a field of positive characteristic $p$ and let 
$1 \le k, n<p$ be distinct integers and let $r=\gcd(k,n)$.
If $h(X,Y)$ is an irreducible polynomial factor of $F_{k,n}(X,Y) \in K[X,Y]$ of degree $d$,
other than a factor of $X^r-1$, $Y^r-1$ or $X^r+Y^r$, then
$$d \ge \max\left\{ \min \left\{ p/k, \sqrt{k/3}-r \right\},   \min \left\{ p/n, \sqrt{n/3}-r \right\}\right\}.
$$
\end{lem}

\begin{proof}
We proceed as in Lemma~\ref{lem:deg-bound n} and consider the curve $\cX$. We define
$u_1=x^n,u_2=y^n,u_3 =1-x^n-y^n$ so that they satisfy the unit equation $u_1+u_2+u_3=1$.
Again, the poles of $u_1,u_2,u_3$ are among the at most $d$ points at infinity of $\cX$ with multiplicity 
at most $n$ and
that $u_1$ (respectively $u_2$) have zeros at the at most $d$ zeros of $x$ (respectively $y$).
As for $u_3$, note that $u_3^{k/r} = (x^k+y^k-1)^{n/r}$, which shows that each zero of $u_3$ has multiplicity
divisible by $n/r$, as $\gcd(k/r,n/r)=1$. Since $u_3$ has degree at most $dn$, it follows that $u_3$ has at most
$dr$ distinct zeros. Hence $u_1,u_2,u_3$ are $S$-units for a set $S$ with $\#S \le (r+3)d$. If $dn < p$ we can
apply the unit equation bound, provided $u_1,u_2,u_3$ are linearly independent over $K$, to get
$\deg u_1 \le 3(d(d-3)+\#S) \le 3(d^2+rd)$. If $u_1$ is not constant, then $\deg u_1 \ge n$ and we get $d \ge \sqrt{n/3}-r$.
If $u_1$ is constant, then $x$ is constant and it can be shown that $h$ is a factor of $X^r-1$, which was excluded.

If $au_1+bu_2+cu_3=0$ and $abc \ne 0$, then we consider the unit equation $-au_1/bu_2 -cu_3/bu_2=1$ and
conclude as before if $u_1/u_2$ is not constant. We note that $abc=0$ means some quotient of two of 
$u_1,u_2,u_3$ is constant. These possibilities are ruled out since they lead to $h$ being a factor of 
$X^r-1$, $Y^r-1$ or $X^r+Y^r$. So we get $d \ge  \min \{p/n, \sqrt{n/3}-r\}$.

Finally, reversing the roles of $n$ and $k$ gives the inequality $d \ge  \min \{p/k, \sqrt{k/3}-r\}$ and
completes the proof.
\end{proof}

\subsection{Upper bounds on the number of zeros of some equations}
\label{sec:UB-zero} 
We  now derive   bounds on
$$
N_{k,n}  = \#\left\{(x,y) \in \F_p^2:~F_{k,n}(x,y)=0\right\}, 
$$
where $F_{k,n}(X,Y) \in  \F_p[X,Y]$ is the polynomial defined by~\eqref{eq:Poly F_kn}.

We start with a bound on
$$
N_{n} =  N_{1,n} 
$$
which is based on  Lemma~\ref{lem:deg-bound n}. 

\begin{thm}
\label{thm:Nn} 
We have,
$$
N_n \ll p +    n^{4/3}p^{2/3}. 
$$
\end{thm}

\begin{proof} Clearly there are 
\begin{equation}
\label{eq:N0}
N_{n}^{(0)} \ll p.
\end{equation}
 points on the on linear factors $X-1$, $Y-1$, $X+Y$
of $F_n$.

Each of the remaining factors is  of degree 
$$
d \gg \min\{p/n, n \}
$$
by Lemma~\ref{lem:deg-bound n}. Hence the number $J$ of such irreducible factors is   
$$
J\ll  \frac{\deg F_{n}}{  \min\{p/n, n \}}  \ll \max\{1, n^2/p\}.
$$  

The contribution  to $N_n$  from each irreducible  factor   $h\mid F_n$  of  degree
$d < p^{1/4}$   is $O(p)$ by the Weil bound (see~\cite{Lor}).
Hence the total contribution $N_{n}^{(1)} $ from such factors 
can be estimated as
\begin{equation}
\label{eq:N1}
N_{n}^{(1)}  \ll  J p \ll  (\max\{1, n^2/p\} p \ll \max\{p, n^2\}.
\end{equation}

Each irreducible factor   $h\mid F_n$  of  degree $\deg h = d\ge p^{1/4}$  
contributes $O\(d^{4/3}p^{2/3}\)$ by~\cite[Theorem~(i)]{V} and, in total they contribute 
\begin{equation}
\begin{split}
\label{eq:N2}
N_{n}^{(2)} & \ll \sum_{\substack{h\mid F_n, \text{irred}\\ \deg h \ge p^{1/4}}}  (\deg h) ^{4/3}p^{2/3} \le \( \sum_{\substack{h\mid F_n, \text{irred}\\ \deg h \ge p^{1/4}}}  \deg h\)^{4/3}p^{2/3} \\
& \le n^{4/3}p^{2/3},
\end{split} 
\end{equation}
using the convexity of the function $z \mapsto z^{4/3}$. 

Combining~\eqref{eq:N0}, \eqref{eq:N1} and~\eqref{eq:N2} we obtain 
$$N_n\le N_{n}^{(0)} +N_{n}^{(1)} +N_{n}^{(2)}  \ll p + n^2 +  n^{4/3}p^{2/3}.
$$
 Since $n^2 \le  n^{4/3}p^{2/3}$ for $n \le p$, the result follows. 
\end{proof}

 \begin{cor}
\label{cor:Tn}
We have, 
$$
T_n \ll p^2 +    n^{4/3}p^{5/3}. 
$$
\end{cor}

\begin{proof} Eliminating $u$ we obtain that $T_n$ is equal to the number of solutions to the 
equation $x^n+y^n = v^n + (x+y - v)^n$. For $v=0$ there are at most $np$ values for $(x,y) \in \F_p^2$.
If $v \ne 0$, then replacing $x\mapsto xv$,   $y\mapsto yv$, we obtain $x^n+y^n = 1 + (x+y - 1)^n$.
Hence, by Theorem~\ref{thm:Nn}, we have  $T_n \le np + pN_n \ll np + p^2 +  n^{4/3}p^{5/3}$.  Since $n \le p$, 
the result follows.
\end{proof}

For an arbitrary our  bound on
$N_{k,n}$   is based on  Lemma~\ref{lem:deg-bound kn}. 

\begin{thm}
\label{thm:Nkn} 
Let 
$1 \le k, n<p$  be distinct integers  and let $ r =\gcd(k,n)$ and assume that $r \le 0.5 \sqrt{n}$. 
Then we have,
$$
N_{k,n} \ll   k\sqrt{n} p/r +   (kn/r)^{4/3}p^{2/3}. 
$$
\end{thm}

\begin{proof} 
Let $s=\gcd(k,n, p-1) = \gcd(r,p-1)$. Clearly there are 
\begin{equation}
\label{eq:N0k}
N_{k,n}^{(0k)} \ll s p.
\end{equation}
 points on the on linear factors  $X^r-1$, $Y^r-1$ or $X^r+Y^r$
of $F_{k,n}$.

Since  $r \le 0.5 \sqrt{n}$, each of the remaining factors is  of degree 
$$d \gg \min\{p/n, \sqrt{n/3}-r\} \gg  \min\{p/n, \sqrt{n}\} 
$$
by Lemma~\ref{lem:deg-bound kn}. Hence, the number $J$ of such irreducible factors is   
$$
J \ll \frac{\deg F_{k,n}}{  \min\{p/n, \sqrt{n}\}} \ll \max\{k\sqrt{n}/r, kn^2/(pr)\}.
$$  

The contribution  to $N_{k,n}$  from each irreducible  factor   $h\mid F_{k,n}$  of  degree
$d < p^{1/4}$   is $O(p)$ by the Weil bound (see~\cite{Lor}).
Hence, similarly to~\eqref{eq:N1},  the total contribution $N_{k,n}^{(1)} $ from such factors 
can be estimated as
\begin{equation}
\label{eq:N1k}
N_{k,n}^{(1)} \ll  J p \ll \max\{k\sqrt{n}/r, kn^2/(pr)\} p \ll \max\{k\sqrt{n} p/r, kn^2/r\}.
\end{equation}

As in the proof of Theorem~\ref{thm:Nn}, we now use that each irreducible factor   $h\mid F_{k,n}$  of  
degree $\deg h = d\ge p^{1/4}$  contributes $O\(d^{4/3}p^{2/3}\)$ by~\cite[Theorem~(i)]{V} and, in total they contribute
similarly as in the bound~\eqref{eq:N2}, using that the degree of $F{k,n}$ is $kn/r$.
\begin{equation}
\label{eq:N2k}
N_{k,n}^{(2)} \ll  (kn/r)^{4/3}p^{2/3},
\end{equation}
using the convexity of the function $z \mapsto z^{4/3}$.

Combining~\eqref{eq:N0k}, \eqref{eq:N1k} and~\eqref{eq:N2k} we obtain 
$$N_{k,n} \le N_{k,n}^{(0)} +N_{k,n}^{(1)} +N_{k,n}^{(2)}  \ll s p + k\sqrt{n} p/r + kn^2/r  +  (kn/r)^{4/3}p^{2/3}.
$$
 Since $r \ll \min\{k, \sqrt{n}\}$  we have
 $$
 s \le r \ll \sqrt{n}\ll  k\sqrt{n}/r 
 $$ 
 and also $n \le p$ we have 
$$ kn^2/r \le n^2 \le  n^{4/3}p^{2/3}\le (kn/r)^{4/3}p^{2/3}.
$$ 
The result now follows. 
\end{proof}

 \begin{cor}
\label{cor:Tkn}
Let 
$1 \le k, n<p$  be distinct integers and let 
$$ r =\gcd(k,n)  \mand   s =\gcd(s,p-1). 
$$ 
Assume that $r  \le 0.5 \sqrt{n}$, then we have,
$$
T_{k,n} \ll  
 k\sqrt{n}sp^2/r +  (kn/r)^{4/3} s p^{5/3}
$$
\end{cor}

\begin{proof}  
Eliminating $u$ we obtain that $T_{k,n} \le s R_{k,n}$ where $R_{k,n}$ is   the number of solutions to the 
equation 
$$(x^n+y^n -v^n)^{k/r} = (x^k+y^k -v^k)^{n/r}.
$$
(as for any fixed $v,x,y$ the power $u^r$ is uniquely defined and so $u$ either $u=0$ or can take at most $s = \gcd(r,p-1)$
values).  

For $v=0$ there are at most $knp/r$ values for $(x,y) \in \F_p^2$.
If $v \ne 0$, then replacing $x\mapsto xv$,   $y\mapsto yv$, we obtain  $(x^n+y^n -1)^{k/r} = (x^k+y^k -1)^{n/r}.$.
Hence, by Theorem~\ref{thm:Nkn}, we have  
 \begin{align*}
T_{k,n} &  \le s R_{k,n} \le s\( knp/r + k\sqrt{n} p^2/r +   (kn/r)^{4/3}p^{5/3} \) \\
& \ll sknp/r +  k\sqrt{n}sp^2/r +  (kn/r)^{4/3} s p^{5/3}.
\end{align*}
Since $r \le k$ and $n \le p$,  we obtain
\begin{equation}
\label{eq:Tkn prelim}
T_{k,n} \ll  
 k\sqrt{n}sp^2/r +  (kn/r)^{4/3} s p^{5/3}
\end{equation}
and the result follows.
\end{proof}

Using the trivial bound $s \le r$ we can simplify Corollary~\ref{cor:Tkn} as 
$$
T_{k,n} \ll  
 k\sqrt{n}p^2 +  (kn)^{4/3} r^{-1/3}  p^{5/3}.
$$
 
 Furthermore,  let   $d$ be the largest 
divisor of $r$ with $\gcd(d,p-1)=1$. If we  set 
$$
k^* = k/d, \qquad n^*=n/d,   \qquad r^* = r/d
$$
then  $r^* = \gcd(k^*, n^*)$ and $\gcd(r^*,p-1) = \gcd(r,p-1) = s$.
Since  $\gcd(d,p-1)=1$ , we clearly have
$T_{k,n}  = T_{k^*,n^*}$. Thus, using~\eqref{eq:Tkn prelim} with $(k^*, n^*, r^*)$ in
place of $(k,n,r)$ we obtain 
 \begin{align*}
T_{k,n} & \ll   k^*\sqrt{n^*}sp^2/r^* +  (k^*n^*/r^*)^{4/3} s p^{5/3}\\
& = k\sqrt{nr^*}sp^2/(r^{3/2}) +  (knr^*/r^2)^{4/3} s p^{5/3}
\end{align*}
 provided that  $r^*  \le 0.5 \sqrt{n}$. Clearly that if $r$ is squarefree that $r^* = s$ in 
 which case we obtain yet another modification of Corollary~\ref{cor:Tkn}\:
 $$
 T_{k,n}   \ll     k\sqrt{n}s^{3/2} p^2r^{-3/2} +  (kn)^{4/3} s^{7/3}  p^{5/3}r^{-8/3}.
 $$

\section{Exponential sums with binomials}

\subsection{Preparations} 
The following relation between $M_{k,n}$ and $T_{k,n}$ has appeared implicitly in several previous 
works. For the sake of completeness we give a short proof. 

\begin{lem}
\label{lem:Mkn vs Tkn}
Let 
$1 \le k, n<p$ be distinct integers and let 
$$s=\gcd(k,n, p-1).
$$
 Then we have,
$$
M_{k,n}^4 \ll  s p  T_{k,n}. 
$$
\end{lem}

\begin{proof} We fix some $a, b\in \F_p^*$. Clearly for any $z \in \F_p^*$ we have
$$
S_{k,n}(a,b) =  \sum_{x=0}^{p-1} \ep\(az^k x^k + bz^n x^n\) = S_{k,n}\(az,bz^n\).
$$
Since $a,b\in \F_p^*$, there are  $(p-1)/s$ pairs $\(az^k,bz^n\)$ that are pairwise distinct. Hence
$$
s^{-1} (p-1) \left|S_{k,n}(a,b) \right|^4 \le \sum_{\lambda, \mu \in \F_p} \left|S_{k,n}(\lambda, \mu) \right|^4 .
$$
By the orthogonality of exponential functions 
$$
\sum_{\lambda, \mu \in \F_p} \left|S_{k,n}(\lambda, \mu) \right|^4  =  p^2 T_n
$$
and the result follows. 
\end{proof}

\subsection{Bounds of exponential sums} 
Combining Corollary~\ref{cor:Tn} with Lemma~\ref{lem:Mkn vs Tkn} (used with $k=1$),  we immediately 
obtain:

\begin{thm}
\label{thm:Mn} 
For $1 \le n < p$, we have,
$$
M_n \ll p^{3/4}  +    n^{1/3}p^{2/3}. 
$$
\end{thm}

Using the bound~\eqref{eq:Aku} for $n > p^{2/5}$ and Theorem~\ref{thm:Mn} 
otherwise, we obtain

 \begin{cor}
\label{cor:Mn_p-1}
For any  $n \mid p-1$, we have 
$$
M_n \ll p^{4/5}. 
$$
\end{cor}

Similarly, Combining Corollary~\ref{cor:Tkn} with Lemma~\ref{lem:Mkn vs Tkn} we derive:

\begin{thm}
\label{thm:Mkn} 
Let 
$1 \le k, n<p$  be distinct integers and let 
$$ r =\gcd(k,n)  \mand   s =\gcd(s,p-1). 
$$ 
Assume that $r  \le 0.5 \sqrt{n}$, then we have,
$$
M_{k,n}  \ll k^{1/4} n^{1/8} s^{1/2}p^{3/4}/r^{1/4} +  (kn/r)^{1/3} s^{1/2} p^{2/3}.
$$
\end{thm}

Again,  using the trivial bound $s \le r$ we derive from Theorem~\ref{thm:Mkn}  that 
$$
M_{k,n}  \ll k^{1/4} n^{1/8} s^{1/4}p^{3/4}  +  (kn)^{1/3} s^{1/6} p^{2/3}.
$$

\section{Comments}

We note that  in Lemma~\ref{lem:deg-bound kn} 
regardless of whether $k < n$ or $ k> n$ both lower bounds can be
of use. However in other results, such as Theorems~\ref{thm:Nkn} and~\eqref{thm:Mkn}, 
without loss of generality we can assume that $k < n$.

A computer calculation for primes $p \le 67$ using Magma~\cite{Magma} verified that, except for $n=(p+1)/2$, 
the polynomials $F_n$, $2 \le n < p$,  have a unique irreducible factor in addition to the trivial linear factors
explicitly given in Lemma \ref{lem:deg-bound n}. For $n=(p+1)/2$, on the other hand, $F_n$ factors completely
into quadratic polynomials in addition to the trivial linear factors. The polynomials $F_{k, n}$,  $2 \le k<n<p \le 29$, however,
all have a unique irreducible factor in addition to the trivial linear factors
explicitly given in Lemma \ref{lem:deg-bound kn}.

We also remark that our approach applies to binomial Laurent polynomials, that is, when 
one of $k$ and $n$ is negative. 

Finally, we expect  that our method can give new results fo exponential sums with trinomials
and other sparse polynomials and Laurent polynomials, see~\cite{Mac,MSS}. 

\section*{Acknowledgements}

During the preparation of this work the first author was  supported   by the ARC Grants~DP170100786 and DP180100201. 
The second author would like to thank UNSW for the hospitality during which part of this work was done.

\bibliographystyle{amsalpha}

\end{document}